\documentclass[11pt]{article}
\usepackage{a4,amssymb,amsthm,amsmath}

\theoremstyle{plain}
\newtheorem{thm}{Theorem}
\newtheorem{lem}{Lemma} 

\theoremstyle{definition} 
\newtheorem{example}{Example}
\newtheorem{prob}{Problem}
\newtheorem{defn}{Definition}
\newtheorem{rem}{Remark}
\newtheorem{vanlem}{Vanishing Lemma}

\newcommand{\nicefrac}[2]{\leave§ vmode\kern.1em
\raise.5ex\hbox{\the\scriptfont0 #1}\kern-.1em
/\kern-.15em\lower.25ex\hbox{\the\scriptfont0 #2}}
\def\halv{\mathchoice{{ \frac 1 2}}{1/2}{1/2}{1/2}}

\newcommand{\C}{{\mathbb C}} 
\newcommand{\R}{{\mathbb R}} 

\newcommand{\norm}[1]{\left \Vert {#1} \right \Vert}

\newcommand{\abs}[1]{{\left| {#1} \right|}} \newcommand{\p}[1]{{\left(
      {#1} \right)}} 

\newcommand{\Oh}[1]{{O \p{#1}}}

\renewcommand{\Re}{\operatorname{Re}} 
\nonstopmode
 \begin{document}
\date{}

\author{Johan Andersson\thanks{Department of Mathematics, Stockholm University, SE-106 91 Stockholm SWEDEN, Email: {\texttt {johana@math.su.se}.}}}

\title{On generalized Hardy classes  of Dirichlet series}

\maketitle

\begin{abstract}
We generalize the Hardy class $H^2$ of Dirichlet series studied by  Hedenmalm, Lindqvist, Olofsson, Olsen, Saksman, Seip and others to consider  more general Dirichlet series. We prove some results on this class, such as estimates for its  logarithmic $ L^1$-norm  in short intervals. We relate this to, and use these results to make  a recent nonvanishing result of Dirichlet series of ours more explicit. In particular we give an application on the Hurwitz zeta-function.\end{abstract}

\tableofcontents

\section{A nonvanishing result for Dirichlet series} 
A  recent result\footnote{First presented at the Tata institute of fundamental research, Mumbai at the international conference in analytic number theory, October 6th 2009} that allow us to prove things like non universality of the Hurwitz zeta-function on the line $\Re(s)=1$, give us lower estimates for means of the Riemann zeta-function close to $\Re(s)=1$, as well as a complete solution to a problem of Ramachandra, see \cite{Andersson4}  is the following:

\begin{vanlem} Any Dirichlet series that is identically zero on an interval of absolute convergence is identically zero on the complex plane. \end{vanlem}

We first remark that unless the interval lies on the Dirichlet series abscissa of convergence, the function is analytic on the interval and the result is immediate. The general case does not turn out to be much more difficult however. We will prove results in this paper that will generalize this result and make it more explicit. The proper class to use for generalizing this is a the class of absolutely convergent Dirichlet series. However, we will also develop  the theory of $H^2$-classes, which is somewhat deeper and is interesting in its own right.

\section{The $H^2$ Hardy class of Dirichlet series}

 We say that the Dirichlet series $L \in H^2$ if
\begin{gather} \label{Lcl}
 L(s)=\sum_{n=1}^\infty a_n n^{-s}, \\ \intertext{and} \label{norm}
 \norm{L}_2^2= \sum_{n=1}^\infty \abs{a_n}^2<\infty.
\end{gather}
This class has been studied by for example Hedenmalm-Saksman \cite{HedSaks}, \, Saksman-Seip \cite{SaksSeip},  \, Hedenmalm-Lindqvist-Seip, \, \cite{HedLinSeip, HedLinSeip2}, Olofsson \cite{Olofsson},  Olsen-Seip \cite{Olsen} and Olsen-Saksman \cite{OlsenSaksman}.

The primary way this class of Dirichlet series has been studied is to identify a Dirichlet series with an infinite dimensional Fourier series $L^2(T^\infty)$, where each variable corresponds to a prime. This is an important method that has also been extensively used in the theory of Universality of zeta-functions, see e.g. Laurin{\v{c}}ikas  \cite{Laurincikas} or Steuding \cite[p. 65]{Steuding}. An example of an important result is the  result of Hedenmalm-Saksman \cite{HedSaks}  of finding an analogue of Carleson's theorem that is valid both when characters are considered as the parameter space as well as when the parameter space has been the variable $s$.

Our method will instead be related to the fact that by the Riemann mapping theorem there is a holomorphic bijection between the half plane and the disc.  Thus we will reduce the study of the Hardy space $H^2$ of Dirichlet series, holomorphic on the half-plane $\Re(s)>1/2$ to the Hardy space $H^2(T)$ of holomorphic functions on the disc. For the theory of $H^2(T)$ see e.g. \cite[Chapter 17]{Rudin}.

 This is an even more classical situation, where a lot of results are known, some of which we can be  transferred to the $H^2$ space of Dirichlet series. The proof we presented in Mumbai for Theorem 1 depended on this bijection and used the fact that the zero-set of a function in the Hardy class $H^2(T)$ has zero measure unless the function is identically zero. In this paper we will also make use of this bijection to move  back to the half-plane to obtain an effective version of Theorem 1. Doing this gives us a version of Jensen's inequality on a half plane. This is essentially done in Koosis \cite[pp. 49-52]{Koosis}, although our result will have a slightly different formulation. 

In this situation it is no longer  natural to use the characters as a parameter space, so those results can not be obtained by this method.  However, the study of a single function $L(s)$ on the line $\Re(s)=\halv$ lends itself natural to this situation. This line $\Re(s)=\halv$ may be the abscissa of convergence, but does not need to be, since even Dirichlet polynomials which are convergent everywhere belongs to $H^2$. We remark that with this normalization it roughly corresponds to the Riemann zeta-function for $\Re(s)=1$, although the Riemann zeta-function $\zeta(s+1/2)$ does not belong to  $H^2$. 

\section{Generalized $H^2$ Hardy classes of Dirichlet series}

We will choose to show our result for a somewhat more general class of Dirichlet series. Let 

\begin{gather} 
       L(s)= \sum_{n=0}^\infty a_n e^{-\lambda_n s}, \label{Ldef} \\ \intertext{and assume that we have the Dirichlet condition} 
0=\lambda_0< \lambda_1 < \lambda_2 \cdots, \label{dircon} \\ \intertext{Furthermore assume that}
 \frac{e^{-\sigma (\lambda_{n}+\lambda_m) }}{\abs{\lambda_{n}-\lambda_m}}  \leq C, \qquad \qquad (n \neq m). \label{Cdef}
\end{gather}
We will define the norm similarly to \eqref{norm}
\begin{gather}
  \label{norm2}
   \norm{L}_2^2=\sum_{n=0}^\infty |a_n|^2.
\end{gather}
We can now define our extended class of Dirichlet series. We will find it convenient to make two definitions.
\begin{defn}
 We say that $L(s)$ belongs to the extended class of Dirichlet series $H^2(\lambda_n,\sigma)$ whenever we have \eqref{Ldef} -  \eqref{norm2}, for some $C>0$ and  the norm 
$\norm{L}_2<\infty$.
\end{defn}
\begin{defn}
 We say that $L(s)$ belongs to the extended class of Dirichlet series $H^2(C,\sigma)$ whenever we have \eqref{Ldef} -  \eqref{norm2}, and  the norm $\norm{L}_2<\infty$.
\end{defn}

\begin{thm}
  Suppose $\lambda_n$ fulfill the Dirichlet condition  \eqref{dircon} and the inequality \eqref{Cdef}. Then  
 \begin{gather*}
    H^2(\lambda_n,\sigma) \subset H^2(C,\sigma).
  \end{gather*}
\end{thm}

\begin{proof}
This follows directly from the Definitions 1 and 2.
\end{proof}
\begin{rem} The Hardy class $H^2(\lambda_n,\sigma)$ is closed under addition, while $H^2(C,\sigma)$ is not. \end{rem}
\subsection{Examples of Dirichlet series from our extended Hardy class}

We give some examples
\begin{example}
 The Hardy class $H^2$ of Dirichlet series defined by \eqref{Lcl} and \eqref{norm} is exactly the Hardy class $H^2(\log (n+1),\halv)$.
\end{example}

\begin{example}
 The classical Hardy class $H^2(T)$ on the circle is exactly the Hardy class $H^2(2 \pi n,0)$.
\end{example}
Thus the two classical Hardy classes are in fact special cases of our extended Hardy-classes of Dirichlet series. For the classical theory of the Hardy class $H^2(T)$ on the circle, see Katznelson \cite{Katznelson} or Rudin \cite{Rudin}. We will find it convenient to also calculate the $C$ in \eqref{Cdef} for these examples.
\begin{lem}
  It is true that  $H^2(\log (n+1),\halv) \subset H^2((\sqrt{2} \log 2)^{-1},1/2),$ where $(\sqrt{2} \log 2)^{-1} = 1.02014$  and that $H^2(2 \pi n,0) \subset H^2(2 \pi,\sigma)$.
\end{lem}

\begin{proof} The first part follows from the fact that
\begin{gather*}
  \frac{((n+1)(n+2))^{-1/2}} {\log(n+2)-\log(n+1)}, \\ 
\intertext{is maximized for $n=0$ when $n \geq 0$ and}
 \frac{((0+1)(0+2))^{-1/2}} {\log(0+2)-\log(0+1)}=\frac 1 {\sqrt 2 \log 2}.
\end{gather*}
The second part is immediate.
\end{proof}

These two examples are in fact the prototypes for the two main cases that can be obtained. The first case also includes for example the case when $\lambda_n \sim  \log n$. The second case includes the case when $\lambda_n \sim n$. 

Except for the natural example $\log n$ in the first class, there seems to be one other natural example that turns up in number theory. if $\lambda_n=\log(n+\alpha)-\log \alpha$ we can almost include the Hurwitz zeta-function $\alpha^s \zeta(s,\alpha)$ in the class $H^2(\log(n+\alpha)-\log \alpha)$. However, it does not quite belong to the Hardy class proper. 
Instead we will consider the following example:
\begin{example}
  The zeta-functions 
  \begin{gather}
   \zeta^{(-z)} (s,\alpha) \alpha^s= \alpha^s \sum_{n=0}^\infty (n+\alpha)^{-s} \p {\log (n+\alpha)}^{-z},
  \end{gather}
 belongs to $H^2(\log(n+\alpha)-\alpha,\halv)$ whenever $\Re(z)>1/2$.
\end{example}
In particular for $z=-1$ it  will just be a primitive function, and on its abscissa of convergence we have
\begin{gather} \label{hurw}
 \zeta^{(-1)}(1+it,\alpha) = -i\int_1^t \zeta(1+ix,\alpha) dx+c_\alpha, \qquad \qquad (t>0),
\end{gather}
where $\zeta(s,\alpha)= \zeta^{(-0)}(s,\alpha)$ denotes the Hurwitz zeta function.

We will also prove the result corresponding  to Theorem 1.
\begin{lem}
 We have that $H^2(\log (n+\alpha)-\log \alpha,\halv) \subset H^2((\sqrt 2 \log 2)^{-1},1/2)$ whenever $0<\alpha<1$.
\end{lem}
\begin{proof}
  The constant $C$ is the greatest for $\alpha=1$, in which case it follows from Lemma 1.
\end{proof}

\begin{rem}
 It is of course true that the Hurwitz zeta-function proper $ \zeta (s,\alpha) \alpha^s$ belongs to the Hardy class $H^2(\log(n+\alpha)-\log \alpha,\sigma)$ for $\sigma>1/2$. However its properties on its abscissa of convergence are deeper than in any half plane $\Re(s) \geq 1+\xi>1$, and we will find use of some variation of this Example later.
\end{rem}

\subsection{Properties of our extended Hardy classes}

\subsubsection{Comparison to the classical $H^2$ class}

Although our extended Hardy classes are more general than the classical Hardy class, they are likely less interesting in general. This compares for example with the relationship with the Hurwitz zeta function and the Riemann zeta function. While the Hurwitz zeta function is a proper generalization of the Riemann zeta function, it is nevertheless less interesting. In fact, since the Hurwitz zeta-function $\zeta(s,\alpha)$ has no Euler product, it lacks the arithmetic content of the Riemann zeta-function and is thus not as important for number theory. For rational arguments, while it does not have an Euler product unless $\alpha=1/2,1$ it does retain arithmetic properties however, since then it can be written as a linear combination of Dirichlet $L$-functions. 

One interesting part of the theory of the Hurwitz zeta function, is that since it is similar to the Riemann zeta function in many ways,  and in fact specializes to the Riemann zeta-function for $\alpha=1$, its theory can improve our understanding of the differences between the arithmetic and non-arithmetic case, and thus lead to a better understanding of the Riemann zeta function itself. 

Similarly, although our extended Hardy class in general lacks important properties of the Hardy classes $H^2$ and $H^2(T)$, such as closure under multiplication of a Dirichlet polynomial,  and therefore lack the arithmetic properties of the classical Hardy classes,   the study of our generalized Hardy classes can hopefully lead to a better understanding of the Hardy class $H^2$ of Dirichlet series itself. Other examples of generalized Hardy classes where some of these arithmetical properties remains are when the integers in $H^2$ are replaced by Beurling generalized integers \cite{Beurling}. We have not studied this case further.

\subsubsection{Linearity}

We will now state some results that will give a connection between different cases of extended Hardy classes of Dirichlet series

\begin{thm}(Linearity)
 Suppose $L(s)$ belong to the Hardy class  $H^2(\lambda_n,\sigma)$, and that $a>0$. Then $L(as)$ belongs to the Hardy class $H^2(a \lambda_n,\sigma/a)$.
\end{thm}

\begin{proof}
 This follows directly from Definition 1 and the proof is immediate.
\end{proof}
\begin{thm}
 Suppose $L(s)$ belong to the Hardy class  $H^2(\lambda_n,\sigma_0)$. Then $L(s)$ belongs to the Hardy class $H^2(\lambda_n,\sigma)$ for any $\sigma \geq \sigma_0$.
\end{thm}

\begin{proof}
  This follows directly from Definition 1 and the proof is immediate.
\end{proof}

\subsubsection{Locally $L^2$}
One important result that is needed in order to develop our theory is to show that the class $H^2(C,\sigma)$ is in fact a Hardy-class is to show that it is locally $L^2$.
\begin{thm}
 Let $L \in H^2(C,\sigma)$. Then
 \begin{gather*}
     \int_0^D \abs{L(\sigma_1+it)}^2dt \leq \p{D+3 \pi C} \norm{L}_2^2, \qquad (\sigma_1 \geq \sigma).
 \end{gather*}
\end{thm}

\begin{proof} The result follows quite easily from the Montgomery-Vaughan inequality in the same way as for classical Dirichlet series. We first prove it for $\sigma_1>\sigma$ and then take the limit $\sigma_1 \to \sigma$.

  By expanding $\abs{L(\sigma_1+it)}^2=L(\sigma_1+it) L(\sigma_1-it)$ as a double sum we see  that
 \begin{gather} \label{aj} 
\int_0^D \abs{L(\sigma_1+it)}^2  dt= \int_0^D \sum_{n,m=0}^\infty a_n a_m  e^{-\sigma_1 (\lambda_n+\lambda_m)} e^{(\lambda_m-\lambda_n)it} dt.
 \end{gather}
By integrating this term wise, the diagonal  terms $n=m$ gives us the contribution
\begin{gather*}
 D \sum_{n=0}^\infty \abs{a_n}^2 e^{-2 \sigma_1 \lambda_n} \leq  D \sum_{n=1}^\infty \abs{a_n}^2=D \norm{L}_2^2.
\end{gather*}
Similarly, we find that the non-diagonal terms in \eqref{aj} gives us
\begin{gather} \label{ajaj2}
 \sum_n \sum_{m \neq n} \frac{a_n \overline{a_m}e^{-\sigma_1 (\lambda_n+\lambda_m)} \p{e^{(\lambda_m-\lambda_n)iD}-1}} {\lambda_m-\lambda_n}.
\end{gather}
The Montgomery-Vaughan inequality \cite{MontVaug} (See also \cite[p. 21]{Ramachandra}, and \cite{Montgomery}) says that
\begin{gather*}
 \abs{\sum \sum_{m \neq n} \frac{b_n \overline{b_m}}{\lambda_m-\lambda_n}} \leq \frac{3 \pi} 2 \sum_n \abs{b_n}^2 \xi_n^{-1},
 \\ \intertext{whenever $\xi_n=\min_{m \neq n} \abs{\lambda_n-\lambda_m}$. By using this theorem twice, first for  $b_n=a_n e^{-\sigma_1 \lambda_n}$ and then for $b_n=a_n e^{-\sigma_1 (\lambda_n+iD)}$  on \eqref{ajaj2} and using  \eqref{Cdef}, we get the non-diagonal contribution}
3 \pi C {\norm L}_2^2 
\end{gather*}
in Theorem 4.
\end{proof}

\begin{rem} The Montgomery-Vaughan inequality is a variant of an inequality of Hilbert  that has found extensive use in analytic number theory. Hedenmalm-Lindqvist-Seip \cite{HedLinSeip}, coming from a different branch of analysis, used a different argument to prove Locally $L^2$ and were unaware of the previous work done by analytic number theorists. See the discussion in \cite{HedLinSeip2}. \end{rem}

 \section{Absolutely convergent Dirichlet series}

\begin{defn} 
  We say that $L(s)$ is absolutely convergent on $\Re(s)=\sigma$  whenever $L(s)$ is defined by \eqref{Ldef}, satisfies the Dirichlet condition \eqref{dircon}, and we have for the  norm $\norm{L}_1$ that
\begin{gather}
  \label{norm3}
   \norm{L}_1=\sum_{n=0}^\infty |a_n| n^{-\sigma} <\infty.
\end{gather}
\end{defn}
\begin{rem}
Each Dirichlet series $L(s) \in H^2(C,\sigma_0)$ is absolutely convergent Dirichlet series on $\Re(s)=\sigma$ for $\sigma>\sigma_0$.  Theorem 9  will give some (rather weak) relationship between the norms.
\end{rem}

 Many corresponding properties for the absolutely convergent Dirichlet series compared to the $H^2$ class are easier to prove.
Similar to the $H^2$ classes 
we have some simple natural examples. 
\begin{example}
   The Dirichlet series $\zeta^{(-z)}(s,\alpha) \alpha^s$ for $\Re(z)>1$, where $\zeta^{(-z)}(s,\alpha)$ is defined by Example 4  is absolutely convergent on $\Re(s)=1$.
\end{example}

\section{Half plane of convergence}
By the Cauchy-Schwarz inequality it is clear that any Dirichlet series in $H^2(C,\sigma)$ is absolutely convergent for $\Re(s)=\sigma_1$ for any $\sigma_1>\sigma$. We will be interested in making this observation quantitative, by relating the different norms in this region. 
\subsection{Shifted Dirichlet series}
For convenience we define the shifted Dirichlet series.

\begin{defn}
  Let $L(s)$ be any Dirichlet series. We define
  \begin{gather*} 
    L_x(s)=L(s+x).
  \end{gather*}
\end{defn}
\noindent The following result is immediate
\begin{thm}
 We have that the norms  $\norm{L_x}_1$ and $\norm{L_x}_2$ are decreasing
as functions of $x$ whenever they are well defined.
\end{thm}

\subsection{Quantitative estimates in the half plane of convergence}

For absolutely convergent Dirichlet series we have the following inequality. 
\begin{thm}
 Let $L(s)$ be an absolutely convergent Dirichlet series on $\Re(s)=\sigma$.  Then we have that $L_x$ is  absolutely convergent on $\Re(s)=\sigma$ for any $x>0$ and we have that 
\begin{gather*}
  \norm{L_x-a_0}_1 \leq e^{-\lambda_1 x}  \norm{L-a_0}_1. 
\end{gather*}
\end{thm}
\begin{proof}
 Since $\lambda_1 \leq  \lambda_n$  for $n \geq 1$ we have that
\begin{align*}
  \norm{L_x-a_0}_1 &= \sum_{n=1}^\infty e^{-\lambda_n (\sigma+x)} \abs{a_n}, \\ &\leq 
   \sum_{n=1}^\infty e^{-\lambda_1 x}  e^{-\lambda_n  \sigma}  \abs{a_n}, \\ 
   &= e^{-\lambda_1 x}  \norm{L-a_0}_1.
  \end{align*} 
\end{proof}
We will prove a similar result for the  $H^2(C,\sigma)$ class of Dirichlet series. First we prove three lemmas.
\begin{lem}
   Let $\lambda_n$  fulfill \eqref{Cdef}. Then
  \begin{gather*}
    \lambda_1 \geq \begin{cases}  \frac{W(\sigma/C)} {\sigma}, & \sigma>0, \\ \frac 1 C, & \sigma=0, \end{cases}
  \end{gather*}
  where $W(x)$ denote the Lambert $W$-function. 
\end{lem}

\begin{proof}
  From \eqref{Cdef} it is clear that $\lambda_1 \geq \lambda$ when $\lambda$ is the solution to
\begin{gather*}
  e^{-\sigma \lambda}=C \lambda.
\end{gather*}
By the definition of the Lambert $W$-function this equation has the solution
\begin{gather*} 
  \lambda=\frac{W(\sigma/C)} {\sigma},
\end{gather*}
 when $\sigma>0$. If $\sigma=0$ we have the equation $1=C \lambda$ and it follows that $\lambda=1/C$.
  \end{proof}

\begin{lem}
  Let $\lambda_n$  fulfill \eqref{Cdef}. Then
  \begin{gather*}
    \sum_{n=1}^\infty e^{-2(\sigma+x) \lambda_n}  \leq \p{1+\frac {C}{2x}}  e^{-2 \lambda_1 x}, \qquad (x>0).
  \end{gather*}
\end{lem}

\begin{proof}
 We have that
 \begin{align*}
 \sum_{n=1}^\infty e^{-2(\sigma+x) \lambda_n} &= e^{-2 \lambda_1 x} \sum_{n=1}^\infty e^{-2 x (\lambda_n-\lambda_1)} e^{-2\sigma \lambda_n}, \\ &\leq     e^{-2 \lambda_1 x} \sum_{n=1}^\infty e^{-2 x (\lambda_{n}-\lambda_{1})} e^{-\sigma (\lambda_{n}+\lambda_{n-1})},
\end{align*}
since $\lambda_n$ is an increasing sequence. By using the inequality \eqref{Cdef} for $n\geq 2$ this can be estimated by
\begin{gather*}
    e^{-2 x \lambda_1}\p{1+ \sum_{n=2}^\infty C (\lambda_{n}-\lambda_{n-1}) e^{-2 x (\lambda_n-\lambda_1)}},
\end{gather*}
Now let \begin{gather}
   a_n=C(\lambda_{n}-\lambda_{n-1}), \qquad \qquad (n \geq 2).
 \end{gather}
We get that
\begin{gather} \label{UI}
1+\sum_{n=2}^\infty C (\lambda_{n}-\lambda_{n-1}) e^{-2 x \lambda_n} = 1+\sum_{n=2}^\infty a_n e^{- 2x/C \sum_{k=2}^n a_k}.
\end{gather}
The right most sum is a lower Riemann sum for the integral
\begin{gather*}
  \int_0^\infty e^{-2xt/C} dt =\frac{C}{2x},
\end{gather*}
from which the lower bound in the Lemma follows.
\end{proof}

\begin{lem}
  Let $L \in H^2(\lambda_n, \sigma)$. Then for $x>0$ we have that
  \begin{gather*}
    \norm{L_x-a_0} \leq  \sqrt{1+\frac{C}{2x}} \, \norm{L-a_0}_2 e^{- \lambda_1 x}.
  \end{gather*}
\end{lem}
\begin{proof}
  This follows  by the triangle inequality,
\begin{align*}
 \norm{L_x-a_0} &=\abs{\sum_{n=1}^\infty e^{-\lambda_n(\sigma+x)}}, \\ 
                &\leq \sum_{n=1}^\infty \abs{a_n} e^{-\lambda_n(\sigma+x)}, \\ 
   \\ \intertext{the Cauchy-Schwarz inequality} &\leq 
 \sqrt{\sum_{n=1}^\infty \abs{a_n}^2   \sum_{n=1}^\infty e^{-2 \lambda_n(\sigma+x)}}, \\ \intertext{and Lemma 4}
       &\leq \norm{L-a_0}_2  \sqrt{1+\frac{C}{2x}} e^{- \lambda_1 x}.
\end{align*}
\end{proof}

\begin{thm}
 Let $L \in H^2(C,\sigma) $. Then we have that $L_x$ is  absolutely convergent on $\Re(s)=\sigma$ for any $x>0$ and we have that 
\begin{gather*}
  \norm{L_x-a_0}_1 \leq \sqrt{1+\frac{C}{2x}} \,   \norm{L-a_0}_2 e^{-Kx}, \intertext{where}
  K= \begin{cases}  \frac{W(\sigma/C)} {\sigma}, & \sigma>0, \\ \frac 1 C, & \sigma=0, \end{cases}
\end{gather*}
where $W(x)$ denote the Lambert $W$-function.
\end{thm}
\begin{proof} This follows from Lemma 3 and Lemma 5. \end{proof}

\subsection{Some sharper results in the case of Classical Dirichlet series}
 Theorems 6 and 7 have sharper variants in the case of classical Dirichlet series since then we know that $\lambda_1=\log 2$. For absolutely convergent classical Dirichlet series we have the following theorem.
\begin{thm}
 Let $L(s)$ be an absolutely convergent classical Dirichlet series on $\Re(s)=\sigma$.  Then we have that $L_x$ is  absolutely convergent on $\Re(s)=\sigma$ for any $x>0$ and we have that
\begin{gather*}
  \norm{L_x-a_0}_1 \leq 2^{-x}  \norm{L-a_0}_1. 
\end{gather*}
\end{thm}
\begin{proof} This follows from Theorem 6. \end{proof}
For the Hardy class $L \in H^2(\log(n+1),\sigma)$ which by Example 2 coincide with the classical $H^2$ Hardy class of Dirichlet series when $\sigma=1/2$ we have the following theorem. 
\begin{thm}
 Let $L \in H^2(\log(n+1),\sigma)$. Then we have that $L_x$ is  absolutely convergent on $\Re(s)=\sigma$ for any $x>0$ and we have that 
\begin{gather*}
  \norm{L_x-a_0}_1 \leq  2^{-x}\sqrt{1+\frac{1}{x\sqrt 8 \log 2}} \,  \norm{L-a_0}_2.
\end{gather*}
\end{thm}
\begin{proof} This follows from Theorem 7 and Lemma 1. \end{proof}

\section{The logarithmic integral}
\subsection{Jensen's inequality in a half-plane}

We will first prove a theorem on the logarithmic integral. We will choose to formulate it rather generally, but we always have our extended class of Dirichlet series in mind, since it will be applied on our classes of Dirichlet series.

\begin{lem}
  Suppose that $\psi$ is a function analytic on $\Re(s) > 1$ such that
  \begin{gather*}
    \sup_{\Re(s)>1} \int_{0}^{1} \abs{\psi(s+it)}^2 dt < \infty,
   \end{gather*}
  and $\psi(2) \neq 0$. Then we have that
  \begin{gather*}
    \int_{-\infty}^\infty \frac {\log^- \abs{\psi \p{1+it}}} {1+t^2} dt \leq     \int_{-\infty}^\infty \frac {\log^+ \psi \abs{\p{\sigma+it}}} {1+t^2} dt - 
      \pi \log  \abs{\psi(2)},
  \end{gather*} 
if $\psi$ is defined on $\Re(s)=1$ by its limit when $\Re(s) \to 1$.
 \end{lem}
\begin{rem}
  We will essentially follow the proof from Koosis \cite{Koosis}, page 49-52. While he does not explicitly state Lemma 6, it follows immediately from his results.  Instead of proving Lemma 6 he uses it to show that
\begin{gather*}
\int_{-\infty}^\infty \frac {\log^+ \abs{\psi \p{1+it}}} {1+t^2} dt <\infty \implies    \int_{-\infty}^\infty \frac {\log^- \psi \abs{\p{\sigma+it}}} {1+t^2} dt <\infty. 
\end{gather*}
This result is also a consequence of  Lemma 6, although we might be forced to use  a shifted variant $\psi(x)=\psi(t+x)$ if $\psi(2)=0$.
\end{rem}

\noindent {\em Proof of Lemma 6.}
 We use the following holomorphic bijection 
 \begin{gather}
   \phi(z)=\frac {2} {(1+z)}, \label{trans} \\ \intertext{which maps the unit disc $\{ z \in \C: |z|<1\}$  to the half-plane $\{s \in \C: \Re(s)>1 \}$. Then}
 f(z)=\psi(\phi(z)), \label{trans2} 
\end{gather}
will be  a holomorphic function on the unit disc, that is $L^2$ on the boundary, and it will belong to the Hardy-class $H^2(T)$ on the circle.\footnote{We used this to prove the Vanishing Lemma in  \cite{Andersson4}. By  \cite[Theorem 17.17]{Rudin} (see also \cite[Theorem 3.14]{Katznelson}) the function $\log |f(z)|$ will be $L^1(T)$ on the circle. Since any interval  on $\Re(s)=1$ is the map of some circle segment it means that $\log \abs{ \psi }$ is locally $L^1$ on the line $\Re(s)=1/2$. From this it follows that the zero-set of $f(e^{i \theta})$ and thus also $\psi(1+it)$  has zero measure for real $t$. } We will  use this bijection to obtain a formula for the logarithmic integral.
 With the change of variables (see e.g. \cite[p. 1]{Koosis})
\begin{gather} \label{i9}
  1+it=\frac{2} {(1+e^{i \theta})}, \\ \intertext{we have  that}
   \theta=\tan \frac t 2, \notag \\ \intertext{and we get that} \label{i10}
 \int_{-\infty}^\infty \frac { \log^{-} \abs{\psi(1+it)}} {1+t^2} dt  = 
  \halv  \int_{0}^{2 \pi} \log^{-} \abs{f(\theta)} d\theta.
\end{gather}
 By Jensen's inequality (Compare with Katznelson \cite[p. 90]{Katznelson})
\begin{gather} \label{jensineq} 
 \log \abs{f(0)} \leq \frac 1 {2\pi} \int_0^{2 \pi} \log \abs{f(re^{i\theta})} d\theta, \\ \intertext{by dividing the logarithm function} \notag
 \log |f(x)|=\log^+|f(x)|+\log^-|f(x)|,
\\ \intertext{and by taking the limit $r\to 1$,  it follows that} \notag
 \log \abs{f(0)} -  \frac 1 {2 \pi} \int_0^{2 \pi} \log^+ \abs{f(e^{i\theta})}d\theta \leq - \frac 1 {2 \pi} \int_0^{2 \pi} \log^{-}\abs{f(e^{i\theta})} d\theta 
\leq 0.
\end{gather}
Rearranging the terms we see that \begin{gather} \label{i909}
 \frac 1 {2 \pi} \int_0^{2 \pi} \log^- \abs{f(e^{i \theta})} d\theta \leq  \frac 1 {2 \pi} \int_0^{2 \pi} \log^{+}\abs{f(e^{i\theta})}d\theta -  \log \abs{f(0)}.
\end{gather}
By \eqref{trans} and \eqref{trans2} it follows that \begin{gather*}
  \log \abs{f(0)}=\log \abs{\psi(2)}.
 \end{gather*}
By using the substitution \eqref{i9} again and the identity \eqref{i10} we translate the integrals in \eqref{i909} back to the $t$-line and we get the result. \qed

\begin{rem}
 We have equality in Lemma 6 if and only if the function $\psi(s)$ has no zeroes for $\Re(s)> 1$. This follows since it is the real part of an analytic function by moving the integration path, and the fact that we have equality in Jensen's inequality \eqref{jensineq} if the function is non-zero on the disc.
\end{rem}

\begin{rem}
  The logarithmic integral that occurs in Lemma 6 has been thoroughly studied in the volumes of Koosis \cite{Koosis,Koosis2}. Other results related to the logarithmic integral, the Paley-Wiener theorems were also important tools in our solution to a generalized problem of Ramachandra \cite{Andersson4}. 
\end{rem}

\subsection{The logarithmic integral for Dirichlet series}
We will now apply Lemma 6 on our Dirichlet series.

\begin{lem}
  Suppose  $L \in H^2(C,\sigma)$ and $L(D+\sigma) \neq 0$ for $D>0$. Then 
  \begin{gather*}
    \frac D \pi \int_{-\infty}^\infty \frac {\log^- \abs{L \p{\sigma+it}}} {D^2+t^2} dt \leq     \frac D \pi \int_{-\infty}^\infty \frac {\log^+ L \abs{\p{\sigma+it}}} {D^2+x^2} dt - 
        \log  \abs{L(D+\sigma)}.
  \end{gather*} 
   \end{lem}

\begin{proof}
  By Theorem 4 it follows that 
\begin{gather*}
    \sup_{\Re(s)>1} \int_{0}^{1} \abs{L(as+b+ait)}^2dt <\infty,
\end{gather*}
whenever $D+b=\sigma$ and $D,b$ are real numbers. By Lemma 6 with 
\begin{gather*}
 \psi(s)=L(Ds+b),
\end{gather*}
we obtain the inequality:
 \begin{gather*}
    \int_{-\infty}^\infty \frac {\log^- \abs{L \p{Dit+b+D}}} {1+t^2} dt \leq    \int_{-\infty}^\infty \frac {\log^+ L \abs{\p{Dit+b+D}}} {1+t^2} dt - 
      \pi \log  \abs{L(2D+b)}.
  \end{gather*} 
The result follows from the substitution $x=t/D$, and  the fact that $D+b=\sigma$.
\end{proof}
We will now apply Theorem 4 and another version of Jensen's inequality 
\begin{thm}
  Suppose  $L \in H^2(C,\sigma)$. Then 
  \begin{gather*}
      \frac D {\pi} \int_{-\infty}^\infty \frac{\log^+ \abs{L(\sigma+it)}}{t^2+D^2} dt \leq  \kappa+\frac 1 2 \log \p{1+\frac{3 \pi C}D}+\log \norm{L}_2, \\ \intertext{whenever the right hand side is non negative, and}
  \kappa= \frac 1 2 \log \p{\tanh \pi + \frac{1}{\pi}} = 0.2735187155.
\end{gather*}
\end{thm}

\begin{proof}
A  variant of a different Jensen's inequality \cite[Theorem 3.3]{Rudin}than the one we have already studied states that
\begin{gather}
 \label{jensineq2}
  \int_{-\infty }^\infty \mu(t) \log^+ \abs{f(t)} dt \leq \log^+ \abs{\int_{-\infty }^\infty \mu(t) |f(t)| dt}, \\ \intertext{whenever} \notag  \int_{-\infty}^\infty \mu(t)dt=1, \qquad \text{and} \qquad  \mu(t) \geq 0.
 \end{gather}
This inequality is true also  if we replace $\log^+$ by any  concave function, for example $-\log^-$. For an application of this inequality for Dirichlet series with Euler product, see our paper \cite{Andersson3}. By the equality
 \begin{gather} \label{yy123} \int_{-\infty}^\infty \frac {dt} {D^2+t^2}= \frac \pi D, \\ \intertext{we can choose $\mu(t)=D/(\pi(t^2+D^2))$ in  \eqref{jensineq2}, and it follows that}
    \frac D {\pi} \int_{-\infty}^\infty \frac{\log^+ \abs{L(\sigma+it)}}{t^2+D^2} dt \leq   \frac 1 2 \log^+ \p{ \frac D {\pi} \int_0^\infty \frac{\abs{L(\sigma+it)}^2}{t^2+D^2} dt}.  \label{ineqa} \end{gather}
We will now use the inequality
\begin{multline*}
 \int_{-\infty}^\infty \frac{\abs{L(\sigma+it)}^2}{t^2+D^2} dt    \leq \int_{-D/2}^{D/2} \abs{L(\sigma+it)}^2 dt +\sum_{k=0}^\infty \frac 1 {D^2((k+1/2)^2+1)}  \times \\  \times  \p{\int_{(k+1/2)D}^{(k+3/2)D} \abs{L(\sigma+it)}^2 dt +  \int_{-(k+1/2)D}^{-(k+3/2)D} \abs{L(\sigma+it)}^2 dt}.
\end{multline*}
By Theorem 4 the integrals can be bounded by $(3 \pi C+D) \norm{L}_2^2$, and  thus
 \begin{multline*} \int_{-\infty}^\infty \frac{\abs{L(\sigma+it)}^2}{t^2+D^2} dt \leq \frac 1 {D^2} \p{1+\sum_{k=0}^\infty \frac 2 {(k+1/2)^2+1}} (3 \pi C+D) \norm{L}_2^2  = \\ =
  \p{1+\pi \tanh \pi} \frac {3 \pi C+D}{D^2} \norm{L}_2^2.
  \end{multline*}
  The result follows by Eq. \eqref{ineqa}.
\end{proof}

\begin{rem}
  The symmetric division of the integral \begin{gather*} 
\int_{-\infty}^\infty \frac{\abs{L(\sigma+it)}^2}{t^2+D^2} dt = \sum_{k=-\infty}^\infty \int_{(k-1/2)D}^{(k+1/2)D} \frac{\abs{L(\sigma+it)}^2}{t^2+D^2} dt, \\ \intertext{gives a somewhat better constant than using the division}
\int_{-\infty}^\infty \frac{\abs{L(\sigma+it)}^2}{t^2+D^2} dt =\sum_{k=-\infty}^\infty\int_{(k-1)D}^{kD} \frac{\abs{L(\sigma+it)}^2}{t^2+D^2} dt. 
\end{gather*}
For this division see e.g. \cite[Eq. (3)]{OlsenSaksman}. If we would have used this instead we would have obtained  $\kappa=1/2 \log(\coth \pi+1/\pi)=0.27918489270$ instead of $\kappa=1/2\log(\tanh \pi+1/\pi)= 0.2735187155$ in Theorem 10. 
\end{rem}

We also prove the corresponding result for the $L^1$-norm defined by Eq. \eqref{norm3}
\begin{lem}
  Suppose  $L(s)$ is absolutely convergent on $\Re(s)=\sigma$. Then
  \begin{gather*}
      \frac D {\pi} \int_{-\infty}^\infty \frac{\log^+ \abs{L(\sigma+it)}}{t^2+D^2} dt \leq  \log^+ \norm{L}_1. 
\end{gather*}
\end{lem}

\begin{proof}
 By the triangle inequality we have that
 \begin{gather*}
    \abs{L(\sigma+it)} = \abs{\sum_{n=0}^\infty a_n e^{-\lambda_n (\sigma+it)}} \leq \sum_{n=0}^\infty \abs{a_n}  e^{-\lambda_n \sigma}=\norm{L}_1. 
  \end{gather*}
  Thus we have that  \begin{gather}
  \label{xxx}
    \log^+\abs{L(\sigma+it)} \leq \log^+ \norm{L}_1. \end{gather}
   Lemma 8 now follows from the identity Eq. \eqref{yy123}.
 \end{proof}
\noindent An immediate  consequence is the following
\begin{thm}
  Suppose  $L \in H^2(C,\sigma)$ and $L(D+\sigma) \neq 0$ for $D>0$. Then we have that
  \begin{gather*}
    \frac D \pi \int_{-\infty}^\infty \frac {\log^- \abs{L \p{\sigma+it}}} {D^2+t^2} dt   \leq \kappa +
      \frac 1 2 \log \p{1+\frac{3 \pi C}D}+\log \norm{L}_2- 
        \log  \abs{L(D+\sigma)},
 \end{gather*}
  where $\kappa$ is defined by Theorem 10.
 \end{thm}
\begin{proof} This follows from Lemma 7 and Theorem 10. \end{proof}

\begin{thm}
  Suppose   $L(s)$ is absolutely convergent on $\Re(s)=\sigma$ and $L(D+\sigma) \neq 0$ for $D>0$. Then we have that
  \begin{gather*}
    \frac D \pi \int_{-\infty}^\infty \frac {\log^- \abs{L \p{\sigma+ix}}} {D^2+x^2} dt  \leq
     \log^+ \norm{L}_1-      \log  \abs{L(D+\sigma)}.
 \end{gather*}
 \end{thm}
\begin{proof} This follows from Lemma 7  and Lemma 8. \end{proof}

\section{Effective non-vanishing results}
\subsection{Requiring that $a_0=1$ and a uniform lower bound for our class}

One type of result we would like to prove when generalizing the Vanishing Lemma is an explicit lower bound for
\begin{gather*}
  \int_T^{T+H} \abs{L(\sigma+it)} dt \geq K>0,
\end{gather*}
valid for all functions in a class $H^2(C,\sigma)$ or the corresponding class of absolutely convergent Dirichlet series, where $K$ only depends on $H$ and the norm $\norm{L}_1$ or $\norm{L}_2$. This is not possible due to the simple example in the classical Hardy class $H^2$ of Dirichlet series\footnote{A similar example can be shown for absolutely convergent Dirichlet series.}
\begin{gather*}
  L_N(s)= \frac 1 {\sqrt N} \sum_{n=N+1}^{2N} (-1)^n n^{-s}.
\end{gather*}
It is clear that $\norm{L_N}_2=1$, but it is also easy to show that
$$
 \lim_{N \to \infty} \int_0^H \abs{L_N\p{\frac 1 2+it}}^2 dt =0,
$$ 
for any $H>0$. Other examples can be given by using Voronin Universality. 

Thus it is clear that we can not find any lower bound that {\em only} depends on the class $H^2(C,\sigma)$. We will however manage to find a bound that depends on two different quantities of a particular Dirichlet series, namely its first coefficients $a_0$ assuming it is non-zero and its norm $\norm{L}_2$.

To simplify the statements of our results and remove their dependence on $a_0$ we will from now on  assume that $a_0=1$ for our Dirichlet series. It is clear that the general case can be transferred to this case.  Assume that $L(s)$ is a Dirichlet series in $H^2(C,\sigma)$ (or absolutely convergent on $\Re(s)=\sigma$)  that is not identically zero. Then there  exists a smallest $k$ such that $a_k \neq 0$. Consider
 $$
 \tilde L(s)=\frac{L(s)}{a_k e^{\lambda_k s}}=1+\sum_{n=1}^\infty
  \frac{a_{n+k}}{a_k} e^{-(\lambda_{n+k}-\lambda_k) s}= \sum_{n=0}^\infty
  \tilde a_n e^{-\tilde \lambda_n s},
$$ 
where $\tilde L(s)$  belong to $H^2(C,\sigma)$ (or absolutely convergent on $\Re(s)=\sigma$) and satisfies $\tilde a_0=1$. General properties for   $L(s)$ can now be deduced from properties for  $\tilde L(s)$.

\subsection{Nonvanishing on a half plane}

We will prove some nonvanishing results for a half plane that follows from Theorems 6,7,8 and 9. 
\begin{lem}
  Let $L_x$ be a Dirichlet series absolutely convergent for $x>0$ on $\Re(s)=\sigma$ with $a_0=1$. Then if
  \begin{gather*}
    \norm{L_x-1}_1 =1-\xi, \qquad (0<\xi<1) \\ \intertext{then}  
     \xi \leq \abs{L(s)} \leq 2-\xi, \qquad (\Re(s) \geq \sigma+x).
  \end{gather*}
\end{lem}
\begin{proof}
 This follows from the triangle inequality.
\end{proof}
\begin{rem}
  If the $\lambda_k$ are linearly independent over  $\mathbb Q$ then Lemma 9 gives the best possible estimate. In particular if $\norm{L_x}_1=2$, then  $$\sup_{L(s)=0} \Re(s)=x+\sigma.$$
\end{rem}

\begin{thm}
 Let $L(s)$ be an absolutely convergent Dirichlet series on $\Re(s)=\sigma$ such that $a_0=1$. Then 
\begin{gather*}
  \xi \leq \abs{L(s+it)} \leq  2-\xi, \qquad \qquad (0<\xi<1),
   \\ \intertext{for}  
   \Re(s) \geq  \sigma+\lambda_1^{-1} \log^+ \frac {\norm{L-1}_1}{1-\xi}.
\end{gather*} 
\end{thm}
\begin{proof}
  This follows from Lemma 9 and Theorem 6.
\end{proof}

The corresponding result for the $H^2(C,\sigma)$ class will be somewhat more complicated since Theorem 7 is more complicated than Theorem 6.

\begin{thm}
 Let ,$L \in H^2(C,\sigma)$ such that $a_0=1$. Then
\begin{gather*}
  \xi \leq \abs{L(s+it)} \leq  2-\xi, \qquad \text{for} \qquad \Re(s) \geq \sigma+ x_\xi \qquad (0<\xi<1),
   \\ \intertext{where $x_\xi$ is the positive solution $x=x_\xi$ to}
     \sqrt{1+\frac C {2x}} e^{-Kx}\norm{L-1}_2 = 1 -\xi.  
  \\ \intertext{In particular we have that}
   x_\xi \leq \max\p{C, K^{-1} \log^+ \frac{\sqrt 3 \norm{L-1}_2} {\sqrt 2(1-\xi)}}
\end{gather*} 
where $K$ is defined by Theorem 7. Furthermore, for a particular Dirichlet series the constant $K$ can be replaced by $\lambda_1$.
\end{thm}

\begin{proof}
   By Theorem 7 we have that
   \begin{gather*}
 \norm{L_x-1}_1 \leq \sqrt{1+\frac {C}{2x}} \norm{L-1}_2 e^{-Kx},
\end{gather*} 
 By Lemma 9 it is now sufficient to prove
 \begin{gather*} 
  \sqrt{1+\frac {C}{2x}} e^{-Kx}  \norm{L-1}_2 \leq 1 -\xi, 
 \end{gather*}
for $x = x_\xi$.  Since the left hand side is a decreasing function in $x$ this implies  the first part of of Theorem 14. Part 2 of Theorem 14 follows from the fact that for $x \geq C$ we have that
$$ \sqrt{1+\frac {C}{2x}}  \leq \sqrt{1+1/2}=\sqrt{3/2}.$$

That $K$ can be replaced by $\lambda_1$ follows by using Lemma 5 instead of Theorem 7.
\end{proof}

\subsection{The logarithm in short intervals}

We will be interested in proving upper estimates for

\begin{lem}
  Assume that $L(s)$ is absolutely convergent Dirichlet series for $\Re(s)> \sigma$. Then
  \begin{gather*}
    \int_0^\delta  \log^{-}\abs{ L(s+it)} dt \leq  \pi \p{D+\frac {\delta^2} {4D}} \times  \frac D {\pi} \int_{-\infty}^\infty \frac{ \log^{-} \abs{L(s-\delta/2+it)}}{D^2+t^2}dt,
  \end{gather*}
  for $\Re(s)>\sigma$.
\end{lem}
\begin{proof} This follows from the fact that $\log^- {\abs{L(s+it)}}$ is a positive function. \end{proof}

\begin{lem}
  Assume that $L(s)$ is  an  absolutely convergent Dirichlet series on $\Re(s)=\sigma$, so that $a_1=1$ and that $\xi \leq \abs{L(s)}$ for $\Re(s)\geq D+\sigma$. Then
  \begin{gather*}
    \int_0^\delta  \log^{-}\abs{ L(s+it)}dt \leq  \pi \p{D+\frac{\delta^2}{4D}} \p{\log \norm{L}_1 - \log \abs{L(\sigma+D+i\delta/2)}}.
\end{gather*}
  for $\Re(s)\geq \sigma$.
\end{lem}
\begin{proof}
  This follows from  Theorem 12 and Lemma 10 and by using the fact that if $a_1=1$ then $\norm{L}_1 \geq 1$ and thus  $\log^+ \norm{L}_1 =\log \norm{L}_1$.
\end{proof}

\begin{lem}
  Assume that $L(s) \in H^2(C,\sigma)$ and that $\xi<\abs{L(s)}$  for $\Re(s)\geq D+\sigma$. Then
  \begin{multline*}
    \int_0^\delta  \log^{-}\abs{ L(s+it)}dt \leq \\ \leq \pi \p{D+\frac{\delta^2}{4D}} \p{\kappa+\log\p{1+\frac{3 \pi C}D}+ \log \norm{L}_2 - \log \abs{L(\sigma+D+i\delta/2)}},
\end{multline*}
  for $\Re(s)\geq \sigma$, where $\kappa$ is defined by Theorem 10.
\end{lem}
\begin{proof}
  This follows from Combining Theorem 11 and Lemma 10 and by using the fact that if $a_1=1$ then $\norm{L}_2 \geq 1$ and thus $\log^+ \norm{L}_2 =\log \norm{L}_2$.
\end{proof}

\begin{thm} Let $L(s)$ be an absolutely convergent Dirichlet series on $\Re(s)=\sigma$ such that $a_0=1$. Then
 \begin{align*}
    i)&  &  \int_T^{T+\delta} \log^{-} \abs{L(\sigma+it)}dt &\leq \pi \p{\frac{(\log \norm{L}_1+\log 2)^2}{\lambda_1}+\delta^2 \lambda_1}, \\
 ii)& & \int_T^{T+\delta} \log^+  \abs{L(\sigma+it)}dt &\leq \delta \log \norm{L}_1. 
\end{align*}
\end{thm}
\begin{proof}
$i)$ By Theorem 13  with $\xi=1/2$ \footnote{This choice of $\xi$ is somewhat arbitrary and chosen since it allows a simple treatment. A more optimal choice of $\xi$ can be found with more work.} and the fact that $\norm{L-1}_1 \leq \norm{L}_1$, we see that that
\begin{gather} \label{uiy}
 -\log \abs{L(\sigma+s)} \leq \log(1/2), \qquad \text{for} \qquad
   \Re(s) \geq D+\sigma, \\ \intertext{where} \label{rest}
   D=\frac {\log{\norm{L}_1-\log(1/2)}} {\lambda_1}  =\frac {\log{\norm{L}_1+\log 2}} {\lambda_1}. 
\end{gather}
 Applying Lemma 11  we obtain
\begin{gather*}
    \int_0^\delta  \log^{-}\abs{ L(s+it)}dt \leq  \pi \p{D+\frac{\delta^2}{4D}} \p{\log \norm{L}_1 - \log (1/2)}.
\end{gather*}
We notice that the last parenthesis equals $\lambda_1 D$ and the result follows from simplifying. 

$ii)$. This follows from Eq. \eqref{xxx} in the same way as Lemma 8.
 \end{proof}

\begin{thm} Let $L \in H^2(C,\sigma)$ and $a_0=1$. Then
 \begin{align*}
 i)&  &     \int_T^{T+\delta} \log^{-} \abs{L(\sigma+it)}dt &\leq \pi \p{\frac{(\log \norm{L}_2+1)^2} K+ \delta^2 K}, \\
 ii)&  & \int_T^{T+\delta} \log^{+} \abs{L(\sigma+it)}dt &\leq \delta \log\p{1+\frac {3 \pi C}{\delta}}+\delta \log \norm{L}_2, \end{align*}
where $K$ is defined as in Theorem 7.Furthermore $K$ may be replaced by $\lambda_1$.
\end{thm}

\begin{proof}
 $i)$ By Theorem 14  with the choice $\xi=1-\sqrt 3/(e \sqrt 2)$ \footnote{This choice of $\xi$ is again somewhat arbitrary - we might calculate a more optimal $\xi$ in a later version of this paper} and the fact that $\norm{L-1}_2 \leq \norm{L}_2$, we see that that
\begin{gather} \notag 
 -\log \abs{L(\sigma+s)} \leq -\log\p{1-\frac{\sqrt 3}{e \sqrt 2}}=0.549, \, \, \, \text{for} \, \, \,
   \Re(s) \geq \max(D,C)+\sigma, \\ \intertext{where} \label{rest555}
   D=\frac {\log{\norm{L}_1+1 }} {K}. 
\end{gather}
We remark that by the definition of $K$ and properties of Lambert's $W$-function\footnote{clear but should maybe find reference}, see Lemma 3 it is clear that
\begin{gather*} 
  \frac 1 K \geq \frac 1 C,
\end{gather*} 
and thus $D \geq C$ and  $\max(C,D)=D$.  Applying Lemma 12  we obtain
\begin{gather} \label{hui}
    \int_0^\delta  \log^{-}\abs{ L(s+it)}dt \leq  \pi \p{D+\frac{\delta^2}{4D}} \p{\kappa +\log \norm{L}_2 +0.549}.
\end{gather}
where $\kappa=0.274$ is defined by Theorem 10. Calculation shows that
$$ \kappa+0.549 =0.822  \leq 1.$$
Thus the integral \eqref{hui} can be bounded by
 $$\p{D+\frac{\delta^2}{4D}} \p{\norm{L}_2 +1}=K \p{D^2+\frac{\delta^2} 4}.$$
The fact that  $K$ may be replaced by $\lambda_1$ follows from the fact that  $K$ may be replaced by $\lambda_1$ in Theorem 14.

$ii)$ This follows from Theorem 4 and Jensen's inequality \eqref{jensineq2}.
\end{proof}

\subsection{Sup-norm in short intervals}

Theorems 15 and 16 implies immediately the corresponding results for sup-norm.
\begin{thm} Let $L(s)$ be an absolutely convergent Dirichlet series on $\Re(s)=\sigma$ such that $a_0=1$. Then
 \begin{gather*}
    \inf_T \max_{t \in [T,T+\delta]} \abs{L(\sigma+it)}\geq \exp\p{-\pi \p{\frac{(\log \norm{L}_1+\log 2)^2}{\lambda_1 \delta}+\delta \lambda_1}}.
\end{gather*}
\end{thm}

\begin{thm} Let $L \in H^2(C,\sigma)$ and $a_0=1$. Then
 \begin{gather*}
    \inf_T \max_{t \in [T,T+\delta]} \abs{L(\sigma+it)}  \geq \exp \p{-\pi \p{\frac{(\log \norm{L}_2+1)^2} {K \delta}+ K \delta}},  
\end{gather*}
where $K$ is defined as in Theorem 7. Furthermore $K$ may be replaced by $\lambda_1$.
\end{thm}

\subsection{$L^p$-norm case in short intervals}
By using Jensen's inequality, Eq. \eqref{jensineq2} on $$
f(t)=\abs{L(\sigma+it)}^p$$ we also get $L^p$-norm variants of Theorems 15 and 16.

\begin{thm}Let $L(s)$ be an absolutely convergent Dirichlet series on $\Re(s)=\sigma$ such that $a_0=1$. Then 
 \begin{gather*}
     \inf_T \p{\frac 1 \delta \int_T^{T+\delta} \abs{L(\sigma+it)}^p dt}^{1/p}  \geq \exp \p{-\pi \p{\frac{(\log \norm{L}_2+\log 2)^2} {\lambda_1 \delta}+  \lambda_1 \delta}}. 
\end{gather*}
\end{thm}

\begin{thm} Let $L \in H^2(C,\sigma)$ and $a_0=1$. Then
 \begin{gather*}
     \inf_T \p{\frac 1 \delta \int_T^{T+\delta} \abs{L(\sigma+it)}^p dt}^{1/p}  \geq \exp \p{-\pi \p{\frac{(\log \norm{L}_2+1)^2} {K \delta}+  K \delta}}, 
\end{gather*}
where $K$ is defined as in Theorem 7. Furthermore $K$ may be replaced by $\lambda_1$.
\end{thm}

\section{Effective non vanishing results for Dirichlet series with bounded coefficients}

We will give   sharper  estimates for Dirichlet series with bounded coefficients. Thus we will not only assume that $L \in H^2(C,\sigma)$ and $a_0=1$, but also assume that $|a_n| \leq 1$, i.e. the coefficients are bounded.
We first show a Lemma that corresponds to Theorem 14.
\begin{lem}
  Suppose  $L(s)$ is a Dirichlet series such that $(3)$, $(4)$ and $(5)$ are true. Suppose also that $a_0=1$ and $|a_n| \leq 1$. Then
\begin{gather*}
  \abs{L(s)} \geq \xi \qquad \p{\Re(s) \geq x_\xi}, \\ \intertext{where $x=x_\xi$ is the positive solution to}
 \p{1+\frac C x}e^{-\lambda_1 x}=1-\xi. \\ \intertext{In particular we have that}
x_\xi \geq \max \p{C,\lambda_1^{-1}\log^+\frac{2} {1-\xi}}. 
\end{gather*}
\end{lem}
 \begin{proof} The first part follows from Lemma 4 and the triangle inequality. The second part follows from the fact that $1+C/x \leq 2$ for $x \geq C$.
 \end{proof}.

\subsection{Logarithm in short intervals}
Our results will follow from Lemma 13 and Lemma 11,12 in the same way as Theorem 15, and 16.
\begin{thm} Let $L(s)$ be a Dirichlet series in  $H^2(C,\sigma)$ and $a_0=1$, and $|a_n|\leq 1$. Then
 \begin{gather*}
    \int_T^{T+\delta} \log^- \abs{L(\sigma+it)} dt \leq K_0+ K_1\log \norm{L}_2, \\ \intertext{where}
   K_0= \pi \p{\max\p{C,\frac{\log 4}{K}} +\frac{\delta^2} {4C}}, \qquad K_1=C_0 K_0,\\ \intertext{and}
 C_0=\frac 1 2 \log\p{\tanh \pi+\frac 1 {\pi}}+\log(1+3 \pi)+\log 2=3.174092008,
\end{gather*}
where $K$ is defined as in Theorem 7. Furthermore $K$ may be replaced by $\lambda_1$.
\end{thm}
\begin{proof}
This follows from using $\xi=1/2$ in Lemma 13 and Lemma 12.\footnote{This gives {\em some} but not give the best choice of constants $(K_0,K_1)$. To get a better choice of $K_0$ at the expense of $K_1$ we need to choose $\xi$ close to $0$. For any particular Dirichlet series even better choice will be obtained by Lemma 9.}
\end{proof}

A similar result is true for absolutely convergent Dirichlet series. However to use this method of proof we also need to assume that $\lambda_n$ fulfill 
\eqref{Cdef}  for some $C>0$

\begin{thm} Let $L(s)$ be a Dirichlet series that is absolutely convergent on $\Re(s)=\sigma$  such that $(3)$, $(4)$ and $(5)$ are true,   $a_0=1$, and $|a_n|\leq 1$. Then
 \begin{gather*}
    \int_T^{T+\delta} \log^- \abs{L(\sigma+it)} dt \leq K_0 \p{\log 2+\log \norm{L}_1}.
\end{gather*}
where $K_0$ is defined by Theorem 21.
\end{thm}
\begin{proof}
This follows from using $\xi=1/2$ in Lemma 13 and Lemma 12.
\end{proof}

\subsection{$L^p$-norm}
Similarly to we have the following results in $L^p$-norm,
\begin{thm} Let $L \in H^2(C,\sigma)$,  $a_0=1$ and $|a_n|\leq 1$. Then
 \begin{gather*}
     \inf_T   \p{\frac 1 \delta \int_T^{T+\delta} \abs{L(\sigma+it)}^p dt}^{1/p}  \geq  (24 \norm{L}_2)^{-K_0/\delta}, 
\end{gather*}
where $K_0$ is defined as in Theorem 21.
\end{thm}
\begin{proof}
  This follows from Theorem 21, Jensen's inequality \eqref{jensineq2} and the fact that $\exp(C_0)=23.9 \leq 24$.
\end{proof}

\begin{thm} Let $L(s)$ be a Dirichlet series that is absolutely convergent on $\Re(s)=\sigma$  such that $(3)$, $(4)$ and $(5)$ are true,   $a_0=1$, and $|a_n|\leq 1$. Then
 \begin{gather*}
     \inf_T \p{\frac 1 \delta \int_T^{T+\delta} \abs{L(\sigma+it)}^p dt}^{1/p}  \geq  (2 \norm{L}_1)^{-K_0/\delta},
\end{gather*}
where $K_0$ is defined as in Theorem 21.\end{thm}
\begin{proof} This follows from Theorem 22 and Jensen's inequality \eqref{jensineq2}. \end{proof}

\subsection{sup-norm}
Similarly to we have the following results in sup-norm,

\begin{thm} Let $L \in H^2(C,\sigma)$,  $a_0=1$ and $|a_n|\leq 1$. Then
 \begin{gather*}
      \inf_T   \max_{t \in[T,T+\delta]} \abs{L(\sigma+it)}  \geq  (24 \norm{L}_2)^{-K_0/\delta}, 
\end{gather*}
where $K_0$ is defined as in Theorem 21.
\end{thm}
\begin{proof}
  This follows from Theorem 21, Eq. \eqref{xxx} and the fact that $\exp(C_0)=23.9 \leq 24$.
\end{proof}

\begin{thm} Let $L(s)$ be a Dirichlet series that is absolutely convergent on $\Re(s)=\sigma$  such that $(3)$, $(4)$ and $(5)$ are true,   $a_0=1$, and $|a_n|\leq 1$. Then
 \begin{gather*}
      \inf_T   \max_{t \in[T,T+\delta]} \abs{L(\sigma+it)}   \geq  (2 \norm{L})^{-K_0/\delta},
\end{gather*}
where $K_0$ is defined as in Theorem 21.\end{thm}
\begin{proof} This follows from Theorem 22 and Eq. \eqref{xxx}. \end{proof}

\subsection{Applications on Dirichlet-Hurwitz series}

Since we will be mainly interested in applications on the Hurwitz zeta-function and classical Dirichlet series we will state some results on Dirichlet-Hurwitz series. We will choose to state the result as a Lemma since it will have applications on the Hurwitz zeta-function on the line $\Re(s)=1$. The result for the $H^2(C,\sigma)$ will give stronger results than the absolutely convergent case. While we could have used Theorem 23, we will choose to use Lemma 9 more directly since it will give stronger results for particular Dirichlet series. The nice thing with e.g. Theorem 23 is that it gives a uniform bound for the whole class $H^2(\sigma,C)$.

\begin{lem}
  Assume that $0 <\alpha \leq 1$, and  that $|a_n| \leq 1$. Then we have for $0 < \delta \leq 0.05$ that
  \begin{gather*}
    \int_0^\delta \abs{\alpha^{-1-it}+\sum_{n=1}^\infty a_n (n+\alpha)^{-1-it}} dt \geq   \alpha^{-1} \p{1 + \alpha \sum_{n=1}^\infty \frac{\abs{a_n}^2} {n+\alpha}}^{-\frac{29} {25 \delta}}e^{-\frac{16} \delta} .
  \end{gather*}
 \end{lem}
\begin{proof}
  Let \begin{gather*}
   L^\alpha(s)=1+\alpha^{s+1/2} \sum_{n=1}^\infty a_n (n+\alpha)^{-s-1/2}.
 \end{gather*}
 We have that $L^\alpha(s) \in H^2(\log(n+\alpha)-\log(\alpha),1/2)$  if the sum on the right hand side in the Lemma is finite. We assume this since otherwise the statement is trivially true (the right hand side equals zero).   Choose \begin{gather} D= 0.7378. \end{gather}
Numerical investigations show that 
 \begin{gather*}
   \zeta\p{1+ D}=\zeta(1.7378)=1.98357\leq 1.98358. 
 \end{gather*}
 It is clear that if
    $$A^{\alpha}(s)=\zeta\p{s+1/2,\alpha} \alpha^s$$
   then $
         \norm{A^{\alpha}_x}_1
        $
        is an increasing function for $0<\alpha<1$ for each $x>0$ and thus takes its maximum for $\alpha=1$. It is also clear that $\norm{L_{x}^\alpha}_1 \leq \norm{A_x^\alpha}_1$ and  by Lemma 9 this implies that
  \begin{gather*}
   0.01642  \leq 2-\zeta(1+D)   \leq \abs{L^\alpha(s)} \leq \zeta(1+D) \leq  1.98358, \, \, \, \, \, \, (\Re(s) \geq 1+D). 
  \end{gather*}
   By Lemma 12 we have that 
   \begin{multline} \label{ort}
     \int_0^{\delta} \log^{-} L^\alpha\p{\frac 1 2+it}  \leq \pi \p{D+\frac{\delta^2} {4D}} \log \norm{L^\alpha} +  \\ +  \pi \p{D+\frac{\delta^2} {4D}} \p{\kappa+\log \p{1+\frac{3 \pi C} D}-\log (2-\zeta(1+D))  }.
   \end{multline}
The terms on the right hand side are maximized for $\delta=0.05$ whenever $0<\delta \leq 0.05$. In particular, the last term can be estimated by
\begin{gather*}
\pi \p{D+\frac{0.05^ 2} {4D}} \p{\kappa+\log \p{1+\frac{3 \pi C} D}-\log (2-\zeta(1+D))}=15.976 \leq 16, 
\end{gather*}
where we have used the value of $\kappa$ defined by Theorem 10, and  the value of $C$ given by Lemma 2.
Furthermore the last term on the right hand side of \eqref{ort} can be estimated by using the fact that 
$$
 \pi D \p{1+\frac{0.05^2}{4D}}=2.3198\leq 2.32=\frac {58} {25}.
$$
Together, these estimates implies that for $0<\delta \leq 0.05$ we have that
$$
 \int_0^{\delta} \log^{-} L^\alpha(1/2+it) dt \leq \frac{58} {25} \log \norm{L^\alpha}_2+16.
$$
By Jensen's inequality \eqref{jensineq2}  we obtain our result.
\end{proof}
\begin{rem}
  The best possible constant in the exponent in Lemma 14 (that is valid for all $0<\alpha\leq 1)$ that we can obtain by this method is $\pi (\zeta^{-1}(2)-1)/2+\epsilon=1.145 \ldots+\epsilon<1.16=29/25$ for any $\epsilon>0$. Then however the constant $16$ will be larger and depend on $\epsilon$. In another direction,  when $\alpha<1$, for small positive values of $\alpha$ the constants $16$ and $1.16$ can be considerably improved. The values of the constants can also be improved if for example we assume that $0<\delta \leq 0.01$ instead of $0<\delta \leq 0.05$.
\end{rem}

We remark that we do not immediately get estimates for the  Hurwitz zeta-function itself by this method since for the Hurwitz zeta-function proper on the line $\Re(s)=1$ the sum on the right hand side will be divergent. One naive attempt is to estimate the Hurwitz zeta-function by a finite Dirichlet-Hurwitz polynomial. However, while this is possible, we will not get a lower bound that is independent of $T$, since the length of the polynomial will depend on $T$. In the next section we will use a convolution argument to get a lower bound independent of $T$

\section{The Hurwitz and Lerch zeta-functions}

We will first prove a lemma from which our result for the Hurwitz zeta-function will follow immediately
\begin{lem}
  Assume that $0 <\alpha \leq 1$, and that $|a_n| \leq 1$. Then  we have for $0<\delta \leq 0.05$ that
  \begin{gather*}
    \inf_{\sigma>1,T} \int_T^{T+\delta} \abs{\alpha^{-\sigma-it}+\sum_{n=1}^\infty a_n (n+\alpha)^{-\sigma-it}} dt \geq \alpha^{-1} \p{1+\frac \alpha \delta}^{-\frac{7}{6\delta}} 10^{-\frac 9 \delta}.  
  \end{gather*}
 \end{lem}

\begin{proof}
  Let
   $$A(s,\alpha)= 1+ \alpha^s \sum_{n=1}^\infty a_n (n+\alpha)^{-s}.$$
  We let   $\Phi \in C_0^\infty(\R)$ be a positive function that has support on $[0,1/175]$ such that $\hat \Phi(0)=1$. Furthermore it is clear that we may assume that
\begin{gather} \label{rrr}
 0 \leq \Phi(x) \leq 176.
\end{gather}
By the definition of the Fourier transform and taking the derivative under the integral sign it is  clear that 
\begin{gather} \label{rrr2}
 \abs{\hat \Phi'(x)} \leq \frac{1}{175},  \qquad \abs{\hat \Phi(x)} \leq 1, \qquad  (x \in \R).
\end{gather} 
Under these assumptions we have the convolution
\begin{align*}
   A_\delta(s,\alpha)&= \int_0^{\delta/175} \Phi(t/\delta) A(s+it,\alpha)dt, \\ &=
\int_{-\infty}^\infty \Phi(t/\delta)A(s+it,\alpha)dt, \\
              &=\alpha^{-s}+\sum_{n=1}^\infty b_n (n+\alpha)^{-s}, 
\end{align*}
where
 $$b_n= a_n \hat \Phi(\delta (\log(n+\alpha)-\log \alpha)).$$
By Lemma 14 and the fact that $175/174 \cdot \frac{29}{25}=7/6$ and $175/174 \cdot 16=1400/87$ we find that
\begin{gather} \label{AbbA}
  \int_0^{174\delta/175} \abs{A_\delta(s+it)}dt \geq  \alpha^{-1} \p{1 + \alpha \sum_{n=1}^\infty \frac{\abs{b_n}^2} {n+\alpha}}^{-7/(6\delta)} e^{-1400/(87 \delta)}.
\end{gather}
We calculate
\begin{gather} \label{kkk23}
\begin{split} \sum_{n=1}^\infty \frac{\abs{b_n}^2} {n+\alpha} &= \sum_{n=1}^\infty \frac{\abs{\hat \Phi(\delta \log(n+\alpha))a_n}^2} {n+\alpha},  \\ &\leq \sum_{n=1}^\infty \frac{\abs{\hat \Phi(\delta (\log(n+\alpha)-\log \alpha))}^2} {n+\alpha}, \\ &=\sum_{n=1}^\infty f(n),
\end{split}
\end{gather} 
where
$$f(x)=\frac{\abs{\hat \Phi(\delta (\log(n+\alpha)-\log \alpha))}^2}{n+\alpha}.$$
By taking the derivative of $f(x)$ and using the inequalities \eqref{rrr2} we find that
$$
 \abs{\hat f(x)} \leq \frac{1+2/175 \delta} {(n+\alpha)^{2}} \leq \frac{1.02} {(n+\alpha)^{2}}. \qquad (\delta<1).
$$
By  the following complex variant of the mean value theorem
\begin{gather*}
   \abs{\int_a^{a+1} f(x) dx -f(a)} 
\leq \frac 1 2 \max_{x \in [a,a+1]} |f'(x)|, 
\end{gather*}
and the fact that
\begin{gather*}
  \sum_{n=1}^\infty \frac{1.02} {(n+\alpha)^{2}} \leq 1.02\zeta(2) \leq 2,
\end{gather*}
we find that the sum \eqref{kkk23} can be bounded by the integral
$$ \int_1^\infty \frac{\abs{\hat \Phi(\delta \log x)}^2} x dx + 1.$$
 By the substitution $y=\delta \log x$ we see that
\begin{gather*}
  \int_1^\infty \frac{\abs{\hat \Phi(\delta \log x)}^2} x dx = \frac 1 \delta \int_0^\infty  \abs{\hat \Phi(y)}^2 dy. \\ \intertext{Since $\Phi$ is a real valued function it follows that $|\hat \Phi(x)|=|\hat \Phi(-x)|$. By the Plancherel identity we find that the integral equals}
    \frac 1 {2\delta} \int_0^{1/2} \abs{\Phi(x)}^2 dx.
\end{gather*}
Thus we have that
\begin{gather*}
 \sum_{n=1}^\infty \frac{\abs{b_n}^2} {n+\alpha} \leq  \frac 1 {2\delta} \int_0^{1/175} \abs{\Phi(x)}^2 dx + 1.
\end{gather*}
From \eqref{rrr} and the fact that $\Phi$ has support on $[0,1/175]$ it follows that
$$
 \sum_{n=1}^\infty \frac{\abs{b_n}^2} {n+\alpha} \leq  \frac {176^2} {2 \cdot 175 \delta}  + 1 \leq \frac {90} {\delta}, \qquad (\delta<0.05).
$$
Lemma 15 follows from \eqref{AbbA} and the inequalities
\begin{gather} \label{Ab1}
 1+ \alpha  \sum_{n=1}^\infty \frac{\abs{b_n}^2} {n+\alpha} \leq  \frac 1 {90} \p{1+ \frac \alpha \delta }, 
\\ \intertext{and the fact that}  \label{Ab2}
 90^{175/175} e^{1400/85}=9.0 \cdot 10^8 \leq 10^9.
\end{gather}
\end{proof}

Since the Hurwitz zeta-function is continuous up to its abscissa of convergence (except for a pole at $\Re(s)=1$) an immediate consequence of Lemma 15  is the  lower bound in the following Theorem:

\begin{thm} 
 For the Hurwitz zeta-function and  $0<\alpha \leq 1$ we have the following result
\begin{gather*}
\alpha^{-1} \p{1+\frac \alpha \delta}^{-\frac{7}{6\delta}} 10^{-\frac 9 \delta} \leq   \inf_T  \int_{T}^{T+\delta} \abs{\zeta(1+it,\alpha)}dt.
\end{gather*}
\end{thm}
Similarly, the Lerch zeta-function (for its theory see e.g the monograph of ~Garunk{\v{s}}tis-Laurin{\v{c}}ikas \cite{GarLau}) defined by 
$$
 \phi(\alpha,\beta;s)=\sum_{n=0}^\infty e^{2 \pi i n \alpha} (n+\beta)^{-s}, \qquad \qquad (0<\alpha,\beta \leq 1)
$$
for $\Re(s)>1$ and by analytic continuation elsewhere, is continuous up to $\Re(s)=1$ except for a possible pole at  $s=1$. Thus Lemma 15 also implies a similar theorem for this case:
\begin{thm} 
 For the Lerch zeta-function $\phi(\alpha,\beta;s)$  and  $0<\alpha,\beta \leq 1$  we have  the following result \begin{gather*}
 \beta^{-1} \p{1+\frac \beta \delta}^{-\frac{7}{6\delta}} 10^{-\frac 9 \delta}  \leq  \inf_T  \int_{T}^{T+\delta} \abs{\phi(\alpha,\beta;1+it)}dt.
\end{gather*}
\end{thm}
We will also state versions of these inequalities when we take the infimum of $\alpha$ and $\beta$. First we state two variants Lemma 15. Since the classical Dirichlet series case is of special interest we state a version for this case:
\begin{lem}
  Assume   that $|a_n| \leq 1$. Then  we have for $0<\delta \leq 0.05$ that
  \begin{gather*}
    \inf_{\sigma>1,T}  \int_T^{T+\delta} \abs{1+\sum_{n=1}^\infty a_n n^{-\sigma-it}} dt \geq \delta^{\frac{7}{6\delta}} 10^{-\frac 9 \delta}.  
  \end{gather*}
 \end{lem}
\begin{proof} This is proved in the same way as Lemma 15, and follows from  the error thrown away when using \eqref{Ab1} and \eqref{Ab2}. In particular, in Lemma 15 we could have stated the right hand side as
\begin{gather} \label{Ab3}
 \alpha^{-1} \p{\frac 1 {90}+\frac \alpha \delta}^{-\frac 7 {6 \delta}} (9 \cdot 10^8)^{-\frac 1 \delta},
\end{gather}
from which Lemma 16 follows.
\end{proof}

We also choose to state the case when we take the infimum with respect to $0<\alpha \leq 1$:
\begin{lem}
  Assume that  that $|a_n| \leq 1$. Then  we have for $0<\delta \leq 0.05$ that
  \begin{gather*}
    \inf_{\sigma>1,T} \inf_{0<\alpha \leq 1} \int_T^{T+\delta} \abs{\alpha^{-\sigma-it}+\sum_{n=1}^\infty a_n (n+\alpha)^{-\sigma-it}} dt \geq \delta^{\frac{7}{6\delta}} 10^{-\frac 9 \delta}.  
  \end{gather*}
 \end{lem}
\begin{proof} This follows by the proof of Lemma 15.
It is clear that the sum $$ \alpha \sum_{n=1}^\infty \frac{b_n}{n+\alpha}$$ is maximized when $\alpha=1$, in which case the right hand side is minimized in Lemma 15. In the same way as in the proof of Lemma 16 the result follows by using the fact that a sharper variant \eqref{Ab3} of Lemma 15 in fact holds by the same proof method.
 \end{proof}
As consequence of Lemma 17 we obtain the following results for the Hurwitz and Lerch zeta-functions.

\begin{thm} 
 For the Hurwitz zeta-function  we have the following result
\begin{gather*}
 \delta^{\frac{7}{6\delta}} 10^{-\frac 9 \delta} \leq   \inf_{0< \alpha \leq 1} \inf_T  \int_{T}^{T+\delta} \abs{\zeta(1+it,\alpha)}dt \leq \frac{e^{-\gamma} \pi^2}{24} \delta^2 + \Oh{\delta^4}. \qquad (0<\delta \leq 0.05)
\end{gather*}
\begin{proof} The lower bound follows from Lemma 17 and the upper bound from the fact that for $\alpha=1$ the Hurwitz zeta-function equals the Riemann zeta-function and  \cite[Theorem 3]{Andersson3}. \end{proof}
\end{thm}
\begin{thm} 
 For the Lerch zeta-function $\phi(\alpha,\beta;s)$   we have  the following result \begin{gather*}
 \delta^{\frac{7}{6\delta}} 10^{-\frac 9 \delta}  \leq  \inf_T \inf_{0< \alpha,\beta \leq 1} \int_{T}^{T+\delta} \abs{\phi(\alpha,\beta;1+it)}dt \leq\frac{e^{-\gamma} \pi^2}{24} \delta^2 + \Oh{\delta^4},  \, \,  (0<\delta \leq 0.05)
\end{gather*}
\end{thm}
\begin{proof} The lower bound follows from Lemma 17 and the upper bound from the fact that for $\alpha,\beta=1$ the Lerch zeta-function equals the Riemann zeta-function and  \cite[Theorem 3]{Andersson3}. \end{proof}
Although Theorem 27 gives an explicit version of  \cite[Corollary 3]{Andersson4}, it is still a weak estimate compared to what we  proved for the Riemann zeta-function \cite[Theorem 3]{Andersson3}.
\begin{gather}  \label{aui}
 \inf_T \int_{T}^{T+\delta} \abs{\zeta(1+it)}dt =\frac{e^{-\gamma} \pi^2}{24} \delta^2 + \Oh{\delta^4}, \\
 \label{auui} \inf_T \int_{T}^{T+\delta} \abs{\zeta(1+it)}^{-1}dt = \frac{e^{-\gamma}}{4} \delta^2 + \Oh{\delta^4}.  
\end{gather}
 To prove this result we also used a version of Jensen's inequality \eqref{jensineq2}. However, the fact that the Riemann zeta function possess an Euler product allowed us to get a substantially sharper result. In fact the result is  surprisingly sharp since it turned out that a function that minimized a logarithmic $L^1$ norm in fact was approximately constant and thus also minimized the $L^p$ norm.

We remark that also the results in our paper \cite{Andersson4} can be made effective with this 
result, although it requires some small work. We will do this in 
\cite{Andersson5}. Our result for the Riemann zeta-function suggests the following problems.
\begin{prob}
 Determine asymptotic estimates for the quantities in Theorem 29 and Theorem 30. The problem may also be considered for fixed $\alpha,\beta$ (Theorems 27 and 28). Will it be different if $\alpha$ is rational, algebraic irrational or transcendent?
\end{prob}
This problem seems quite difficult. A simpler variant is the following.
\begin{prob}   
 Can the lower bound in Theorems 29 and 30 be replaced by a
 $\delta^N$ for some sufficiently  large $N$? Is $N=2$ possible or can we obtain a better upper bound than something of the order $\delta^2$. Can we at least prove a lower bound of the order $e^{-C/\delta}$?
\end{prob} 
We remark that if we can prove a lower bound of order $\delta^2$ in Problem 2 this would imply that the Hurwitz zeta-function has no double zeroes on $\Re(s)=1$. In contrast with the Riemann zeta-function, it is well known that the Hurwitz zeta-function may have zeroes on $\Re(s)=1$. Equation \eqref{auui} also suggests the following problem
\begin{prob} 
 Find a nontrivial lower bound for every $\delta>0$ for the quantity
\begin{gather*}
 \inf_T\inf_{0< \alpha \leq 1} \int_{T}^{T+\delta} \abs{\zeta(1+it,\alpha)}^{-1}dt.
\end{gather*}
or prove that it equals zero for some or all $\delta>0$.
\end{prob}
Similar problems may be stated for the Lerch zeta-function. The difficulty with the proof method we used to prove Lemma 15 is that it is not likely that the inverse of the Dirichlet series will be a Dirichlet series of our class, and when  we take the convolution with the Dirichlet series directly we smooth out the function and the inequality will be in the wrong direction.
\begin{rem}
 Problems 1,2,3 may also be stated in $L^p$-norm and sup-norm.
\end{rem}
\section{The min-max problem for our general classes of Dirichlet series}
For the general classes of functions similar problems might be stated. For example in sup-norm we have the following problem 
\begin{prob}
 Determine for $M,\delta,\sigma,C>0$ the following quantity
  \begin{gather*}
    \min_{\substack{L \in H^2(C,\sigma), \\ a_0=1, \norm{L}_2=M }} \max_{t \in [0,\delta]} \abs{L(\sigma+it)}.
  \end{gather*}
\end{prob}
Theorem 18 gives a nontrivial lower bound that shows that it is not zero. A special case is the problem for the classical Hardy class $H^2$ of Dirichlet series.
\begin{prob}
 Determine for $M,\delta>0$ the following quantity
  \begin{gather*}
    \min_{L \in H^2, a_0=1, \norm{L}_2=M} \max_{t \in [0,\delta]} \abs{L(1/2+it)}.
  \end{gather*}
\end{prob}
  Let $A(s)=(1-2^{1-s})^n$. It is clear that $L \in H^2$,  and it follows from the binomial theorem that $\norm{L}_2=3^n$. Since $A(s)$ has a zero of order $n$ for $s=1$ it follows that
  $$
   \max_{t \in [0,\delta]} \abs{A(1+it)} \ll \delta^{n}.
  $$
Thus this example gives us the upper bounds
$$
  \delta^{\lfloor \log M/\log 3 \rfloor}
$$
 in Problem 4.
 Similar constructions can be obtained for problem 5 given $\sigma, C>0$ and sufficiently large $M$.

\def\cprime{$'$} \def\polhk#1{\setbox0=\hbox{#1}{\ooalign{\hidewidth
  \lower1.5ex\hbox{`}\hidewidth\crcr\unhbox0}}}
\bibliographystyle{plain}

\begin{thebibliography}{10}

\bibitem{Andersson2}
J.~Andersson.
\newblock Lavrent{\cprime}ev's approximation theorem with nonvanishing
  polynomials and universality of zeta-functions.
\newblock {\em New Directions in Value-distribution Theory of Zeta and
  L-functions: Wurzburg Conference, October 6-10, 2008 (Berichte aus der
  Mathematik)}, pages 7--10, December 31, 2009. 
\newblock arXiv:1010.0386 [math.NT]

\bibitem{Andersson6}
J.~Andersson.
\newblock Nonuniversality on the critical line.
\newblock arXiv:1207.4927 [math.NT]

\bibitem{Andersson4}
J.~Andersson.
\newblock On a problem of {R}amachandra and approximation of functions by
  {D}irichlet polynomials with bounded coefficients.
\newblock arXiv:1207.4624 [math.NT]

\bibitem{Andersson5}
J.~Andersson.
\newblock On the {B}alasubramanian-{R}amachandra method close to {R}e$(s)=1$.
\newblock  Forthcoming.

\bibitem{Andersson3}
J.~Andersson.
\newblock On the zeta function on the line {R}e$(s)=1$.
\newblock arXiv:1207.4336 [math.NT]

\bibitem{Beurling}
A.~Beurling.
\newblock Analyse de la loi asymptotique de la distribution des nombres
  premiers g{\'e}n{\'e}ralis{\'e}s.
\newblock {\em Acta. Math.}, (68):255--291, 1937.

\bibitem{GarLau}
A.~Laurin{\v{c}}ikas and R.~Garunk{\v{s}}tis.
\newblock {\em The {L}erch zeta-function}.
\newblock Kluwer Academic Publishers, Dordrecht, 2002.

\bibitem{HedLinSeip}
H.~Hedenmalm, P.~Lindqvist, and K.~Seip.
\newblock A {H}ilbert space of {D}irichlet series and systems of dilated
  functions in {$L^2(0,1)$}.
\newblock {\em Duke Math. J.}, 86(1):1--37, 1997.

\bibitem{HedLinSeip2}
H.~Hedenmalm, P.~Lindqvist, and K.~Seip.
\newblock Addendum to: ``{A} {H}ilbert space of {D}irichlet series and systems
  of dilated functions in {$L^2(0,1)$}'' [{D}uke {M}ath.\ {J}.\ {\bf 86}
  (1997), no. 1, 1--37; {MR}1427844 (99i:42033)].
\newblock {\em Duke Math. J.}, 99(1):175--178, 1999.

\bibitem{HedSaks}
H.~Hedenmalm and E.~Saksman.
\newblock Carleson's convergence theorem for {D}irichlet series.
\newblock {\em Pacific J. Math.}, 208(1):85--109, 2003.

\bibitem{Katznelson}
Y.~Katznelson.
\newblock {\em An introduction to harmonic analysis}.
\newblock Dover Publications Inc., New York, corrected edition, 1976.

\bibitem{Koosis2}
P.~Koosis.
\newblock {\em The logarithmic integral. {II}}, volume~21 of {\em Cambridge
  Studies in Advanced Mathematics}.
\newblock Cambridge University Press, Cambridge, 1992.

\bibitem{Koosis}
P.~Koosis.
\newblock {\em The logarithmic integral. {I}}, volume~12 of {\em Cambridge
  Studies in Advanced Mathematics}.
\newblock Cambridge University Press, Cambridge, 1998.
\newblock Corrected reprint of the 1988 original.

\bibitem{Laurincikas}
A.~Laurin{\v{c}}ikas.
\newblock {\em Limit theorems for the {R}iemann zeta-function}, volume 352 of
  {\em Mathematics and its Applications}.
\newblock Kluwer Academic Publishers Group, Dordrecht, 1996.

\bibitem{MontVaug}
H.~L. Montgomery and R.~C. Vaughan.
\newblock Hilbert's inequality.
\newblock {\em J. London Math. Soc. (2)}, 8:73--82, 1974.

\bibitem{Montgomery}
H.~L. Montgomery.
\newblock {\em Ten lectures on the interface between analytic number theory and
  harmonic analysis}, volume~84 of {\em CBMS Regional Conference Series in
  Mathematics}.
\newblock Published for the Conference Board of the Mathematical Sciences,
  Washington, DC, 1994.

\bibitem{Olofsson}
A.~Olofsson.
\newblock On the shift semigroup on the {H}ardy space of {D}irichlet series.
\newblock {\em Acta {M}athematica {H}ungarica},  128  (2010),  no. 3, 265--286.

\bibitem{OlsenSaksman}
J-F.~Olsen and E.~Saksman.
\newblock On the boundary behaviour of the hardy spaces of dirichlet series and
  a frame bound estimate.
\newblock Technical Report arXiv:0907.2708, Jul 2009.
\newblock Comments: 28 pages.

\bibitem{Olsen}
J-F.~Olsen and K.~Seip.
\newblock Local interpolation in {H}ilbert spaces of {D}irichlet series.
\newblock {\em Proc. Amer. Math. Soc.}, 136(1):203--212 (electronic), 2008.

\bibitem{Ramachandra}
K.~Ramachandra.
\newblock {\em On the mean-value and omega-theorems for the {R}iemann
  zeta-function}, volume~85 of {\em Tata Institute of Fundamental Research
  Lectures on Mathematics and Physics}.
\newblock Published for the Tata Institute of Fundamental Research, Bombay,
  1995.

\bibitem{Rudin}
W.~Rudin.
\newblock {\em Real and complex analysis}.
\newblock McGraw-Hill Book Co., New York, second edition, 1974.
\newblock McGraw-Hill Series in Higher Mathematics.

\bibitem{SaksSeip}
E.~Saksman and K.~Seip.
\newblock Integral means and boundary limits of {D}irichlet series.
\newblock {\em Bull. Lond. Math. Soc.}, 41(3):411--422, 2009.

\bibitem{Steuding}
J.~Steuding.
\newblock {\em Value-distribution of {$L$}-functions}, volume 1877 of {\em
  Lecture Notes in Mathematics}.
\newblock Springer, Berlin, 2007.

\end{thebibliography}

\end{document}